\documentclass{lms}

\usepackage{amsmath}
\usepackage{amssymb}
\usepackage{algorithm}
\usepackage{algorithmic}
\usepackage{upgreek}
\usepackage[hyphens]{url}
\usepackage{graphicx}
\usepackage{hyperref}

\newtheorem{theorem}{Theorem}
\newtheorem{lemma}[theorem]{Lemma}

\numberwithin{equation}{section}

\begin{document}

\title[Implementation of the HRR formula]{Efficient implementation of the \\ Hardy-Ramanujan-Rademacher formula}
\author{Fredrik Johansson}
\extraline{Supported by Austrian Science Fund (FWF) grant Y464-N18 (Fast Computer Algebra for Special Functions).}

\classno{11Y55 (primary), 11-04, 11P83 (secondary)}

\maketitle

\begin{abstract}
We describe how the Hardy-Ramanujan-Rademacher formula can be implemented to allow the partition function $p(n)$ to be computed with softly optimal complexity $O(n^{1/2+o(1)})$ and very little overhead. A new implementation based on these techniques achieves speedups in excess of a factor 500 over previously published software and has been used by the author to calculate $p(10^{19})$, an exponent twice as large as in previously reported computations.

We also investigate performance
for multi-evaluation of $p(n)$, where our implementation of the
Hardy-Ramanujan-Rademacher formula becomes
superior to power series methods on far denser sets of indices than
previous implementations. As an application, we
determine over 22 billion new congruences for the partition function,
extending Weaver's tabulation of 76,065 congruences.
\end{abstract}

\section{Introduction}

\noindent Let $p(n)$ denote the number of partitions of $n$, or the number of ways that $n$ can be written as a sum of positive integers without regard to the order of the terms (A000041 in \cite{oeis}). The classical way to compute $p(n)$ uses the generating function representation of $p(n)$ combined with Euler's pentagonal number theorem
\begin{equation}
\label{eq:power}
\sum_{n=0}^{\infty} p(n) x^n = \prod_{k=1}^{\infty} \frac{1}{1-x^k} =
\left(\sum_{k=-\infty}^{\infty} (-1)^k x^{k(3k-1)/2}\right)^{-1}
\end{equation}
from which one can construct the recursive relation
\begin{equation}
\label{eq:recursion}
p(n) = \sum_{k=1}^n (-1)^{k+1}\left(
p\left(n-\frac{k(3k-1)}{2}\right)+
p\left(n-\frac{k(3k+1)}{2}\right)\right).
\end{equation}

Equation \eqref{eq:recursion} provides a simple and reasonably efficient way to compute the list of values $p(0), p(1), \ldots, p(n-1), p(n)$. Alternatively, applying FFT-based power series inversion to the right-hand side of \eqref{eq:power} gives an asymptotically faster, essentially optimal algorithm for the same set of values.

An attractive feature of Euler's method, in both the recursive and FFT incarnations, is that the values can be computed more efficiently modulo a small prime number. This is useful for investigating partition function congruences, such as in a recent large-scale computation of $p(n)$ modulo small primes for $n$ up to $10^9$ \cite{calkin}.

While efficient for computing $p(n)$ for all $n$ up to some limit, Euler's formula is impractical for evaluating $p(n)$ for an isolated, large $n$. One of the most astonishing number-theoretical discoveries of the 20th century is the Hardy-Ramanujan-Rademacher formula, first given as an asymptotic expansion by Hardy and Ramanujan in 1917 \cite{harj} and subsequently refined to an exact representation by Rademacher in 1936 \cite{rade}, which provides a direct and computationally efficient expression for the single value $p(n)$.

Simplified to a first-order estimate, the Hardy-Ramanujan-Rademacher formula states that
\begin{equation}
p(n) \sim \frac{1}{4n\sqrt{3}} \, e^{\pi \sqrt{2n/3}},
\label{eq:asymp}
\end{equation}
from which one gathers that $p(n)$ is a number with roughly $n^{1/2}$ decimal digits. The full version can be stated as
\begin{align}
\label{eq:rade1}
p(n) &= \sum_{k=1}^{N} \left(\sqrt{\frac{3}{k}} \, \frac{4}{24n-1}\right) A_k(n) \; U\left(\frac{C(n)}{k}\right) + R(n,N), \\
\label{eq:rade2}
U(x) &= \cosh(x) - \frac{\sinh(x)}{x}, \quad C(n) = \frac{\pi}{6} \sqrt{24n-1}, \\
\label{eq:rade3}
A_k(n) &= \sum_{h=0}^{k-1} \delta_{\gcd(h,k),1}
    \exp\left(\pi i \left[s(h,k) - \frac{2hn}{k} \right]\right)
\end{align}
where $s(h,k)$ is the Dedekind sum
\begin{equation}
\label{eq:dedek}
s(h,k) = \sum_{i=1}^{k-1} \frac{i}{k} \left( \frac{hi}{k} -
\left\lfloor \frac{hi}{k} \right\rfloor - \frac{1}{2} \right)
\end{equation}
and where the remainder satisfies $|R(n,N)| < M(n,N)$ with
\begin{equation}
\label{eq:bound}
M(n,N) = \frac{44 \pi^2}{225 \sqrt{3}} N^{-1/2} +
         \frac{\pi \sqrt{2}}{75} \left(\frac{N}{n-1}\right)^{1/2}
         \sinh\left(\frac{\pi}{N} \sqrt{\frac{2n}{3}}\right).
\end{equation}

It is easily shown that $M(n, cn^{1/2}) \sim n^{-1/4}$ for every positive $c$. Rademacher's bound \eqref{eq:bound} therefore implies that $O(n^{1/2})$ terms in \eqref{eq:rade1} suffice to compute $p(n)$ exactly by forcing $|R(n,N)| < 1/2$ and rounding to the nearest integer. For example, we can take $N = \lceil n^{1/2} \rceil$ when $n \ge 65$.

In fact, it was pointed out by Odlyzko \cite{odlyzko,knuth05} that the Hardy-Ramanujan-Rademacher formula ``gives an algorithm for calculating $p(n)$ that is close to optimal, since the number of bit operations is not much larger than the number of bits of $p(n)$''.  In other words, the time complexity should not be much higher than the trivial lower bound $\Omega(n^{1/2})$ derived from \eqref{eq:asymp} just for writing down the result. Odlyzko's claim warrants some elaboration, since the Hardy-Ramanujan-Rademacher formula ostensibly is a triply nested sum containing $O(n^{3/2})$ inner terms.

The computational utility of the Hardy-Ramanujan-Rademacher formula was, of course, realized long before the availability of electronic computers. For instance, Lehmer \cite{lehmer3} used it to verify Ramanujan's conjectures $p(599) \equiv 0 \bmod 5^3$ and $p(721) \equiv 0 \bmod 11^2$. Implementations are now available in numerous mathematical software systems, including Pari/GP \cite{pari}, Mathematica \cite{wolfram} and Sage \cite{sage}. However, apart from Odlyzko's remark, we find few algorithmic accounts of the Hardy-Ramanujan-Rademacher formula in the literature, nor any investigation into the optimality of the available implementations.

The present paper describes a new C implementation of the Hardy-Ramanujan-Rademacher formula. The code is freely available as a component of the Fast Library for Number Theory (FLINT) \cite{hart}, released under the terms of the GNU General Public License. We show that the complexity for computing $p(n)$ indeed can be bounded by $O(n^{1/2+o(1)})$, and observe that our implementation comes close to being optimal in practice, improving on the speed of previously published software by more than two orders of magnitude. 

We benchmark the code by computing some extremely large isolated values of $p(n)$. We also investigate efficiency compared to power series methods for evaluation of multiple values, and finally apply our implementation to the problem of computing congruences for $p(n)$.

\section{Simplification of exponential sums}

A naive implementation of formulas \eqref{eq:rade1}--\eqref{eq:dedek}
requires $O(n^{3/2})$ integer operations to evaluate Dedekind sums, and $O(n)$ numerical evaluations of complex exponentials (or cosines, since the imaginary parts ultimately cancel out). In the following section, we describe how the number of integer operations and cosine evaluations can be reduced, for the moment ignoring numerical evaluation.

A first improvement, used for instance in the Sage implementation, is to recognize that Dedekind sums can be evaluated in $O(\log k)$ steps using a GCD-style algorithm, as described by Apostol \cite{apostol}, or with Knuth's fraction-free algorithm \cite{knuth} which avoids the overhead of rational arithmetic. This reduces the total number of integer operations to $O(n \log n)$, which is a dramatic improvement but still leaves the cost of computing $p(n)$ quadratic in the size of the final result.

Fortunately, the $A_k(n)$ sums have additional structure as discussed in \cite{lehmer,lehmer2,radewhit,whiteman,hagis}, allowing the computational complexity to be reduced. Since numerous implementers of the Hardy-Ramanujan-Rademacher formula until now appear to have overlooked these results,
it seems appropriate that we reproduce the main
formulas and assess the computational issues in more detail.
We describe two concrete algorithms: one simple, and one asymptotically fast, the latter being implemented in FLINT.

\subsection{A simple algorithm}

Using properties of the Dedekind eta function, one can derive
the formula (which Whiteman \cite{whiteman} attributes to Selberg)
\begin{equation}
A_k(n) = \left(\frac{k}{3}\right)^{1/2}
\!\!\!\!\!\!\!
\sum_{(3l^2+l)/2\equiv-n \bmod k}
\!\!\!\!\!\!\!
(-1)^l \cos\left(\frac{6l+1}{6k} \pi\right)
\label{eq:selb}
\end{equation}
in which the summation ranges over $0 \le l < 2k$ and only $O(k^{1/2})$ terms are nonzero.
With a simple brute force search for solutions of the quadratic equation, this representation provides a way to compute $A_k(n)$ that is both simpler and more efficient than the usual definition \eqref{eq:rade3}.

Although a brute force search requires $O(k)$ loop iterations, the successive quadratic terms can be generated without multiplications or divisions using two coupled linear recurrences. This only costs a few processor cycles per loop iteration, which is a substantial improvement over computing Dedekind sums, and means that the cost up to fairly large $k$ effectively will be dominated by evaluating $O(k^{1/2})$ cosines, adding up to $O(n^{3/4})$ function evaluations for computing $p(n)$.

\begin{algorithm}
\caption{Simple algorithm for evaluating $A_k(n)$}
\label{alg:simplesum}
\begin{algorithmic}
\renewcommand{\algorithmicrequire}{\textbf{Input:}}
\renewcommand{\algorithmicensure}{\textbf{Output:}}
\REQUIRE Integers $k, n \ge 0$
\ENSURE $s = A_k(n)$, where $A_k(n)$ is defined as in \eqref{eq:rade3}
\IF {$k \le 1$}
    \RETURN $k$
\ELSIF {$k = 2$}
    \RETURN $(-1)^n$
\ENDIF
\STATE $(s, r, m) \gets (0, 2, (n \bmod k))$
\FOR {$0 \le l < 2k$}
    \IF {$m = 0$}
        \STATE $s \gets s + (-1)^l \cos\left(\pi(6l+1)/(6k)\right)$
    \ENDIF
\STATE $m \gets m + r$
\STATE \textbf{if} $m \ge k$ \textbf{then} $m \gets m - k$ \COMMENT{$m \gets m \bmod k$}
\STATE $r \gets r + 3$
\STATE \textbf{if} $r \ge k$ \textbf{then} $r \gets r - k$ \COMMENT{$r \gets r \bmod k$}
\ENDFOR
\RETURN $(k/3)^{1/2} \, s$
\end{algorithmic}
\end{algorithm}

A basic implementation of \eqref{eq:selb} is given as Algorithm \ref{alg:simplesum}. Here the variable $m$ runs over the successive values of $(3l^2+l)/2$, and $r$ runs over the differences between consecutive $m$. Various improvements are possible: a modification of the equation allows cutting the loop range in half when $k$ is odd, and the number of cosine evaluations can be reduced by counting the multiplicities of unique angles after reduction to $[0, \pi/4)$, evaluating a weighted sum $\sum w_i \cos(\theta_i)$ at the end -- possibly using trigonometric addition theorems to exploit the fact that the differences $\theta_{i+1}-\theta_i$ between successive angles tend to repeat for many different $i$.

\subsection{A fast algorithm}

From Selberg's formula \eqref{eq:selb}, a still more efficient but considerably more complicated multiplicative decomposition of $A_k(n)$ can be obtained.
The advantage of this representation is that it only
contains $O(\log k)$ cosine factors, bringing the total number of cosine evaluations for $p(n)$
down to $O(n^{1/2} \log n)$. It also reveals exactly when $A_k(n) = 0$ (which is about half the time). We stress that these results are not new; the formulas are given in full detail and with proofs in \cite{whiteman}.

First consider the case when $k$ is a power of a prime. Clearly $A_1(n) = 1$ and $A_2(n) = (-1)^n$. Otherwise let $k = p^{\lambda}$ and $v = 1-24n$. Then, using the notation $(a|m)$ for Jacobi symbols to avoid confusion with fractions, we have
\begin{equation}
A_k(n) =
  \begin{cases}
   (-1)^{\lambda} (-1|m_2) k^{1/2} \sin(4\pi m_2/8k)                            & \text{if } p = 2 \\
   2(-1)^{\lambda+1} (m_3|3) (k/3)^{1/2} \sin(4\pi m_3/3k)   & \text{if } p = 3 \\
   2(3|k) k^{1/2} \cos(4\pi m_p/k)                                              & \text{if } p > 3 \\
  \end{cases}
\label{eq:factor1}
\end{equation}
where $m_2$, $m_3$ and $m_p$ respectively are any solutions of
\begin{align}
(3m_2)^2 & \equiv v  \bmod 8k \\
(8m_3)^2 & \equiv v \bmod 3k \\
(24m_p)^2 & \equiv v \bmod k
\end{align}
provided, when $p > 3$, that such an $m_p$ exists and that $\gcd(v, k) = 1$. If, on the other hand, $p > 3$ and either of these two conditions do not hold, we have
\begin{equation}
A_k(n) =
  \begin{cases}
   0                   & \text{if } v \text{ is not a quadratic residue modulo } k \\
   (3|k) k^{1/2}       & \text{if } v \equiv 0 \bmod p, \lambda = 0 \\
   0                   & \text{if } v \equiv 0 \bmod p, \lambda > 1. \\
  \end{cases}
\label{eq:factor3}
\end{equation}

If $k$ is not a prime power, assume that $k = k_1 k_2$ where $\gcd(k_1, k_2) = 1$.
Then we can factor $A_k(n)$ as $A_k(n) = A_{k_1}(n_1) A_{k_2}(n_2)$, where $n_1, n_2$ are any solutions of the following equations. If $k_1 = 2$, then
\begin{equation}
\begin{cases}
32n_2 \equiv 8n+1 \bmod k_2 \\
n_1 \equiv n-(k_2^2-1)/8 \bmod 2,
\end{cases}
\label{eq:factor2a}
\end{equation}
if $k_1 = 4$, then
\begin{equation}
\begin{cases}
128n_2 \equiv 8n+5 \bmod k_2 \\
k_2^2 n_1 \equiv n-2-(k_2^2-1)/8 \bmod 4,
\end{cases}
\label{eq:factor2b}
\end{equation}
and if $k_1$ is odd or divisible by 8, then
\begin{equation}
\begin{cases}
k_2^2 d_2 e n_1 \equiv d_2 e n + (k_2^2-1)/d_1 \bmod k_1 \\
k_1^2 d_1 e n_2 \equiv d_1 e n + (k_1^2-1)/d_2 \bmod k_2
\end{cases}
\label{eq:factor2}
\end{equation}
where $d_1 = \gcd(24, k_1)$, $d_2 = \gcd(24, k_2)$, $24 = d_1 d_2 e$.

Here $(k^2-1)/d$ denotes an operation done on integers, rather than a modular division.
All other solving steps in \eqref{eq:factor1}--\eqref{eq:factor2} amount to computing greatest common divisors,
carrying out modular ring operations, finding modular inverses, and computing modular square roots.
Repeated application of these formulas results in
Algorithm \ref{alg:fastsum}, where we omit the detailed arithmetic
for brevity.

\begin{algorithm}
\caption{Fast algorithm for evaluating $A_k(n)$}
\label{alg:fastsum}
\begin{algorithmic}
\renewcommand{\algorithmicrequire}{\textbf{Input:}}
\renewcommand{\algorithmicensure}{\textbf{Output:}}
\REQUIRE Integers $k \ge 1$, $n \ge 0$
\ENSURE $s = A_k(n)$, where $A_k(n)$ is defined as in \eqref{eq:rade3}
\STATE Compute the prime factorization $k = p_1^{\lambda_1} p_2^{\lambda_2} \ldots p_j^{\lambda_j}$
\STATE $s \gets 1$
\FOR {$1 \le i \le j$ and \textbf{while} $s \ne 0$}
    \IF {$i < j$}
        \STATE $(k_1, k_2) \gets (p_i^{\lambda_i}, k / p_i^{\lambda_i})$
        \STATE Compute $n_1, n_2$ by solving the respective case of \eqref{eq:factor2a}--\eqref{eq:factor2}
        \STATE $s \gets s \times A_{k_1}(n_1)$
        \COMMENT{Handle the prime power case using \eqref{eq:factor1}--\eqref{eq:factor3}}
        \STATE $(k, n) \gets (k_2, n_2)$
    \ELSE
        \STATE $s \gets s \times A_{k}(n)$ \COMMENT{Prime power case}
    \ENDIF
\ENDFOR
\RETURN $s$
\end{algorithmic}
\end{algorithm}

\subsection{Computational cost}

A precise complexity analysis of Algorithm \ref{alg:fastsum}
should take into account the cost of integer arithmetic.
Multiplication, division, computation of modular inverses, greatest common divisors and Jacobi
symbols of integers bounded in absolute value by $O(k)$
can all be performed with bit complexity $O(\log^{1+o(1)} k)$.

At first sight, integer factorization might seem to pose a problem. We can, however, factor all indices $k$ summed over in \eqref{eq:rade1} in $O(n^{1/2} \log^{1+o(1)} n)$ bit operations. For example, using the sieve of Eratosthenes, we can precompute a list of length $n^{1/2}$ where entry $k$ is the largest prime dividing $k$.

A fixed index $k$ is a product of at most $O(\log k)$ prime powers with exponents bounded by $O(\log k)$. For each prime power, we need $O(1)$ operations with roughly the cost of multiplication, and $O(1)$ square roots, which are the most expensive operations.

To compute square roots modulo $p^{\lambda}$, we can use the Tonelli-Shanks algorithm \cite{tonelli,shanks} or Cipolla's algorithm \cite{cipolla} modulo $p$ followed by Hensel lifting up to $p^{\lambda}$. Assuming that we know a quadratic nonresidue modulo $p$, the Tonelli-Shanks algorithm requires $O(\log^3 k)$ multiplications in the worst case and $O(\log^2 k)$ multiplications on average, while Cipolla's algorithm requires $O(\log^2 k)$ multiplications in the worst case \cite{crand}. This puts the bit complexity of factoring a single exponential sum $A_k(n)$ at $O(\log^{3+o(1)} k)$,
and gives us the following result:

\begin{theorem}
\label{thm:modular}
Assume that we know a quadratic nonresidue modulo $p$ for all primes $p$ up to $n^{1/2}$. Then we can factor all the $A_k(n)$ required for evaluating $p(n)$ using $O(n^{1/2} \log^{3+o(1)} n)$ bit operations.
\end{theorem}

The assumption in Theorem \ref{thm:modular} can be satisfied with a precomputation that does not affect the complexity. If $n_2(p_k)$ denotes the least quadratic nonresidue modulo the $k$th prime number, it is a theorem of Erd\H{o}s \cite{erdos,pollack} that as $x \to \infty$,
\begin{equation}
\frac{1}{\pi(x)} \sum_{p_k \le x} n_2(p_k) \rightarrow \sum_{k=1}^{\infty} \frac{p_k}{2^k} = C < 3.675.
\end{equation}
Given the primes up to $x = n^{1/2}$, we can therefore build a table of nonresidues by testing no more than $(C + o(1)) \pi(n^{1/2})$ candidates. Since $\pi(n^{1/2}) = O(n^{1/2} / \log n)$ and a quadratic residue test takes $O(\log^{1+o(1)} p)$ time, the total precomputation time is $O(n^{1/2} \log^{o(1)} n)$.

In practice, it is sufficient to generate nonresidues on the fly since $O(1)$ candidates need to be tested on average, but we can only prove an $O(\log^{c} k)$ bound for factoring an isolated $A_k(n)$ by assuming the Extended Riemann Hypothesis which gives $n_2(p) = O(\log^2 p)$ \cite{ankeny}.

\subsection{Implementation notes}

As a matter of practical efficiency, the modular arithmetic should be done with as little overhead as possible. FLINT provides optimized routines for arithmetic with moduli smaller than 32 or 64 bits (depending on the hardware word size) which are used throughout; including, among other things, a binary-style GCD algorithm, division and remainder using precomputed inverses, and supplementary code for operations on two-limb (64 or 128 bit) integers.

We note that since $A_k(n) = A_k(n + k)$, we can always reduce $n$ modulo $k$, and perform all modular arithmetic with moduli up to some small multiple of $k$. In principle, the implementation of the modular arithmetic in FLINT thus allows calculating $p(n)$ up to approximately $n = (2^{64})^2 \approx 10^{38}$ on a 64-bit system, which roughly equals the limit on $n$ imposed by the availability of addressable memory to store $p(n)$.

At present, our implementation of Algorithm \ref{alg:fastsum} simply calls the FLINT routine for integer factorization repeatedly rather than sieving over the indices. Although convenient, this technically results in a higher total complexity than $O(n^{1/2+o(1)})$. However, the code for factoring single-word integers, which uses various optimizations for small factors and Hart's ``One Line Factor'' variant of Lehman's method to find large factors \cite{hart11}, is fast enough that integer factorization only accounts for a small fraction of the running time for any feasible $n$. If needed, full sieving could easily be added in the future.

Likewise, the square root function in FLINT uses the Tonelli-Shanks algorithm and generates a nonresidue modulo $p$ on each call. This is suboptimal in theory but efficient enough in practice.

\section{Numerical evaluation}

We now turn to the problem of numerically evaluating \eqref{eq:rade1}--\eqref{eq:rade2} using arbitrary-precision arithmetic, given access to Algorithm \ref{alg:fastsum} for symbolically decomposing the $A_k(n)$ sums. Although \eqref{eq:bound} bounds the truncation error in the Hardy-Ramanujan-Rademacher series, we must also account for the effects of having to work with finite-precision approximations of the terms.

\subsection{Floating-point precision}

We assume the use of variable-precision binary floating-point arithmetic (a simpler but less efficient alternative, avoiding the need for detailed manual error bounds, would be to use arbitrary-precision interval arithmetic). Basic notions about floating-point arithmetic and error analysis can be found in \cite{higham}.

If the precision is $r$ bits, we let $\varepsilon = 2^{-r}$ denote the unit roundoff. We use the symbol $\hat x$ to signify a floating-point approximation of an exact quantity $x$, having some relative error $\delta = (\hat x - x) / x$ when $x \ne 0$. If $\hat x$ is obtained by rounding $x$ to the nearest representable floating-point number (at most 0.5 ulp error) at precision $r$, we have $|\delta| \le \varepsilon$. Except where otherwise noted, we assume correct rounding to nearest.

A simple strategy for computing $p(n)$ is as follows. For a given $n$, we first determine an $N$ such that $|R(n,N)| < 0.25$, for example using a linear search. A tight upper bound for $\log_2 M(n,N)$ can be computed easily using low-precision arithmetic. We then approximate the $k$th term $t_k$ using a working precision high enough to guarantee

\begin{equation}
|\hat t_k - t_k| \le \frac{0.125}{N},
\label{eq:tbound}
\end{equation}

and perform the outer summation such that the absolute error of each addition is bounded by $0.125 / N$. This clearly guarantees $|\hat p(n) - p(n)| < 0.5$, allowing us to determine the correct value of $p(n)$ by rounding to the nearest integer. We might, alternatively, carry out the additions exactly and save one bit of precision for the terms.

In what follows, we derive a simple but essentially asymptotically tight expression
for a working precision, varying with $k$,
sufficiently high for \eqref{eq:tbound} to hold.
Using Algorithm~\ref{alg:fastsum}, we write the term to be evaluated in
terms of exact integer parameters $\alpha, \beta, a, b, p_i, q_i$ as
\begin{equation}
\label{eq:termk}
t_k = \frac{\alpha}{\beta} \frac{\sqrt a}{\sqrt b} \, U\left(\frac{C}{k}\right)
 \prod_{i=1}^m \cos\left(\frac{p_i \pi}{q_i}\right).
\end{equation}

\begin{lemma}
Let $p \in \mathbb{Z}$, $q \in \mathbb{N}^{+}$ and let $r$ be a precision in bits with $2^r > \max(3q, 64)$. Suppose that $\sin$ and $\cos$ can be evaluated on $(0, \pi/4)$ with relative error at most $2 \varepsilon$ for floating-point input, and suppose that $\pi$ can be approximated with relative error at most $\varepsilon$. Then we can evaluate $\cos(p \pi / q)$ with relative error less than $5.5 \varepsilon$.
\label{thm:cos_error}
\end{lemma}

\begin{proof}

We first reduce $p$ and $q$ with exact integer operations so that $0 < 4 p < q$, giving an angle in the interval $(0,\pi/4)$. Then we approximate $x = p \pi / q$ using three roundings, giving $\hat x = x (1+\delta_x)$ where $|\delta_x| \le (1 + \varepsilon)^3 - 1$. The assumption $\varepsilon < 1 / (3 q)$ gives $(q / (q - 1)) (1+\delta_x) < 1$ and therefore also $\hat x \in (0, \pi/4)$.

Next, we evaluate $f(\hat x)$ where $f = \pm \cos$ or $f = \pm \sin$ depending
on the argument reduction. By Taylor's theorem, we have $f(\hat x) = f(x)(1+\delta'_x)$ where
\begin{equation}
|\delta'_x| = \frac{|f(\hat x) - f(x)|}{f(x)} = \frac{x |\delta_x| |f'(\xi)|}{f(x)}
\end{equation}
for some $\xi$ between $x$ and $\hat x$, giving $|\delta'_x| \le (\frac{1}{4} \pi \sqrt 2) |\delta_x|$. Finally, rounding results in
$$\hat f(\hat x) = f(x)(1+\delta) = f(x)(1+\delta'_x)(1+\delta_f)$$ 
where $|\delta_f| \le 2 \varepsilon$. The inequality $\varepsilon < 1/64$ gives $|\delta| < 5.5 \varepsilon$.
\end{proof}

To obtain a simple error bound for $U(x)$ where $x = C / k$, we make the somewhat crude restriction that $n > 2000$.
We also assume $k < n^{1/2}$ and $x > 3$, which are not restrictions: if $N$ is chosen optimally using
Rademacher's remainder bound \eqref{eq:bound}, the maximum $k$ decreases and the minimum $x$ increases with larger $n$. In particular, $n > 2000$ is sufficient with Rademacher's bound (or any tighter bound for the remainder).

We assume that $C$ is precomputed; of course, this only needs to be done once
during the calculation of $p(n)$, at a precision a few bits higher than that of the $k = 1$ term.

\begin{lemma}
Suppose $n > 2000$ and let $r$ be a precision in bits such that $2^r > \max(16 n^{1/2}, 2^{10})$. Let $x = C / k$ where $C$ is defined as in \eqref{eq:rade2} and where $k$ is constrained such that $k < n^{1/2}$ and $x > 3$. Assume that $\hat C = C(n)(1+\delta_C)$ has been precomputed with $|\delta_C| \le 2 \varepsilon$ and that $\sinh$ and $\cosh$ can be evaluated with relative error at most $2 \varepsilon$ for floating-point input. Then we can evaluate $U(x)$ with relative error at most $(9x+15) \varepsilon$.
\label{thm:U_error}
\end{lemma}

\begin{proof}
We first compute $\hat x = x(1+\delta_x) = (C/k)(1+\delta_C)(1+\delta_0)$ where $|\delta_0| \le \varepsilon$. Next, we compute
\begin{equation}
\hat U(\hat x) = U(\hat x)(1+\delta_U) =
 U(x)(1+\delta'_x)(1+\delta_U) = U(x)(1+\delta)
\label{eq:deltaUbound}
\end{equation}
where we have to bound the error $\delta'_x$ propagated in the composition as well as the rounding error $\delta_U$ in the evaluation of $U(\hat x)$.
Using the inequality $x |\delta_x| < 4 x \varepsilon < \log 2$, we have
\begin{equation}
|\delta'_x|
\le \frac{x |\delta_x| U'(x + x |\delta_x|)}{U(x)}
\le \frac{x |\delta_x| \exp(x + x |\delta_x|)}{2 U(x)}
\le \frac{x |\delta_x| \exp(x)}{U(x)}
\le 3 x |\delta_x|.
\label{eq:delta2bound}
\end{equation}

Evaluating $U(\hat x)$ using the obvious sequence of operations results in
\begin{equation}
|\delta_U| = \frac{\left|\left(\cosh(\hat x)(1+2\delta_1) -
        \dfrac{\sinh(\hat x)}{\hat x}(1+2\delta_2)(1+\delta_3)\right)(1+\delta_4) - U(\hat x)\right|}{U(\hat x)} \\
\end{equation}
where $|\delta_i| \le \varepsilon$ and $\hat x > z$ where $z = 3(1-4\varepsilon)$.
This expression is maximized by setting $\hat x$ as small as possible
and taking $\delta_1 = \delta_4 = -\delta_2 = -\delta_3 = \varepsilon$, which gives

\begin{equation}
|\delta_U| < \frac{\cosh(z)}{U(z)} \varepsilon (3+2 \varepsilon) +
\frac{\sinh(z)}{z \, U(z)} \varepsilon (2 + \varepsilon - 2 \varepsilon^2) < 5.5 \varepsilon
\label{eq:delta3bound}
\end{equation}

Expanding \eqref{eq:deltaUbound} using \eqref{eq:delta2bound} and \eqref{eq:delta3bound} gives $|\delta| < \varepsilon (5.5 + 9 x + 56 x \varepsilon + 33 x \varepsilon^2)$. Finally, we obtain $5.5 + 56 x \varepsilon + 33 x \varepsilon^2 < 15$ by a direct application of the assumptions.
\end{proof}

Put together, assuming floating-point implementations of standard transcendental functions with
at most 1~ulp error (implying a relative error of at most $2\varepsilon$), correctly rounded arithmetic and the constant $\pi$, we have:

\begin{theorem}
Let $n > 2000$. For \eqref{eq:tbound} to hold, it is sufficient to evaluate \eqref{eq:termk} using a precision of $r = \max(\log_2 N + \log_2 |t_k| + \log_2 (10x + 7 m + 22) + 3, \frac{1}{2} \log_2 n + 5, 11)$ bits.
\label{thm:prec}
\end{theorem}

\begin{proof}
We can satisfy the assumptions of lemmas \ref{thm:cos_error} and \ref{thm:U_error}.
In particular, $3q \le 24k < 24 n^{1/2} < 2^r$.
The top-level arithmetic operations in \eqref{eq:termk}, including the square roots, amount to a maximum of $m + 6$ roundings. Lemmas \ref{thm:cos_error} and \ref{thm:U_error} and elementary inequalities give
the relative error bound
\begin{align}
|\delta|
&<
\left(1+\varepsilon\right)^{m+6}
\left(1 + 5.5 \varepsilon\right)^m
\left(1 + (15 + 9x) \varepsilon\right) - 1 \\
&<
\left(1 + \frac{(m+6) \varepsilon}{1 - (m+6)\varepsilon} \right)
\left(1 + \frac{5.5 m \varepsilon}{1 - 5.5 m \varepsilon}\right)
\left(1 + (15 + 9x) \varepsilon\right) - 1 \\
&= \frac{21 \varepsilon + 6.5 m \varepsilon - 33 m \varepsilon^2 - 5.5 m^2 \varepsilon^2 + 9 x \varepsilon}
{(1 - 5.5 \varepsilon m)(1 - \varepsilon (m + 6))} \\
&< (10 x + 7 m + 22) \varepsilon.
\end{align}
The result follows by taking logarithms in \eqref{eq:tbound}.
\end{proof}

To make Theorem \ref{thm:prec} effective, we can use $m \le \log_2 k$ and bound $|t_k|$ using \eqref{eq:rade1} with $|A_k| \le k$ and $U(x) < e^x / 2$, giving
\begin{equation}
\log |t_k| < \frac{(24n-1)^{1/2} \, \pi}{6 k} + \frac{\log k}{2} - \log(24n-1) + \left(\log 2 + \frac{\log 3}{2}\right).
\label{eq:tbound2}
\end{equation}
Naturally, for $n \le 2000$, the same precision bound can be verified to be sufficient through direct computation. We can even reduce overhead for small $n$ by using a tighter precision, say $r = |t_k| + O(1)$, up to some limit small enough to be tested exhaustively (perhaps much larger than 2000). The requirement that $r > \frac{1}{2} \log_2 n + O(1)$ always holds in practice if we set a minimum precision; for $n$ feasible on present hardware, it is sufficient to never drop below IEEE double ($53$-bit) precision.

\subsection{Computational cost}

We assume that $r$-bit floating-point numbers can be multiplied in time $M(r) = O(r \log^{1+o(1)} r)$. It is well known (see \cite{mca}) that the elementary functions exp, log, sin etc. can be evaluated in time $O(M(r) \log r)$ using methods based on the arithmetic-geometric mean (AGM). A popular alternative is binary splitting, which typically has cost $O(M(r) \log^2 r)$ but tends to be faster than the AGM in practice.

To evaluate $p(n)$ using the Hardy-Ramanujan-Rademacher formula, we must add $O(n^{1/2})$ terms each of which can be written as a product of $O(\log k)$ factors. According to \eqref{eq:tbound2} and the error analysis in the previous section, the $k$th term needs to be evaluated to a precision of $O(n^{1/2}/k) + O(\log n)$ bits. Using any combination of $O(M(r) \log^{\alpha} r)$ algorithms for elementary functions, the complexity of the numerical operations is
\begin{equation}
O\left(
\sum_{k=1}^{n^{1/2}}
\log k\,\,
M\!\left(\frac{n^{1/2}}{k}\right)
\log^{\alpha} \frac{n^{1/2}}{k}
\right)
= O\left(n^{1/2} \log^{\alpha+3+o(1)} n\right)
\label{eq:complexity}
\end{equation}
which is nearly optimal in the size of the output. Combined with the cost of the factoring stage, the complexity for the computation of $p(n)$ as a whole is therefore, when properly implemented, softly optimal at $O(n^{1/2+o(1)})$. From \eqref{eq:complexity} with the best known complexity bound for elementary functions, we obtain:

\begin{theorem}
The value $p(n)$ can be computed in time $O(n^{1/2} \log^{4+o(1)} n)$.
\end{theorem}

A subtle but crucial detail in this analysis is that the additions in the main sum must be implemented in such a way that they have cost $O(n^{1/2}/k)$ rather than $O(n^{1/2})$, since the latter would result in an $O(n)$ total complexity. If the additions are performed in-place in memory, we can perform summations the natural way and rely on carry propagation terminating in an expected $O(1)$ steps, but many implementations of arbitrary-precision floating-point arithmetic do not provide this optimization.

One way to solve this problem is to add the terms in reverse order, using a precision that
matches the magnitude of the partial sums. Or, if we add the terms in forward
order, we can amortize the cost by keeping separate
summation variables for the partial sums of terms not exceeding
$r_1, r_1 / 2, r_1 / 4, r_1 / 8, \ldots$ bits.


\subsection{Arithmetic implementation}

FLINT uses the MPIR library, derived from GMP, for arbitrary-precision arithmetic, and the MPFR library on top of MPIR for asymptotically fast arbitrary-precision floating-point numbers and correctly rounded transcendental functions \cite{mpir,gmp,mpfr}.
Thanks to the strong correctness guarantees of MPFR, it is
relatively straightforward to write a provably correct implementation
of the partition function using Theorem \ref{thm:prec}.

Although the default functions provided by MPFR are quite fast, order-of-magnitude speedups were found possible with custom routines for parts of the numerical evaluation. An unfortunate consequence is that our implementation currently relies on routines that, although heuristically sound, have not yet been proved correct, and perhaps are more likely to contain implementation bugs than the well-tested standard functions in MPFR.

All such heuristic parts of the code are, however, well isolated, and we expect that they can be replaced with rigorous versions with equivalent or better performance in the future.

\subsection{Hardware arithmetic}

Inspired by the Sage implementation, which was written by Jonathan Bober, our implementation switches to hardware (IEEE double) floating-point arithmetic to evaluate \eqref{eq:termk} when the precision bound falls below 53 bits. This speeds up evaluation of the ``long tail'' of terms with very small magnitude.

Using hardware arithmetic entails some risk. Although the IEEE
floating-point standard implemented on all modern hardware
guarantees 0.5 ulp error for arithmetic operations, accuracy may be lost,
for example, if the compiler generates long-double instructions which trigger
double rounding, or if the rounding mode of the processor has been changed.

We need to be particularly concerned about the accuracy of transcendental
functions. The hardware transcendental functions on the Intel Pentium processor and its
descendants guarantee an error of at most 1~ulp when rounding to
nearest \cite{intel}, as do the software routines in the portable and widely
used FDLIBM library \cite{fdlibm}. Nevertheless, some systems may
be equipped with poorer implementations.

Fortunately, the bound \eqref{eq:bound} and Theorem \ref{thm:prec}
are lax enough in practice that errors up to a few ulp can be tolerated, and we expect
any reasonably implemented double-precision transcendental
functions to be adequate. Most importantly, range reducing the
arguments of trigonometric functions to $(0, \pi/4)$
avoids catastrophic error for large arguments which is a
misfeature of some implementations.


\subsection{High-precision evaluation of exponentials}

MPFR implements the exponential and hyperbolic functions using binary splitting at high precision, which is asymptotically fast up to logarithmic factors. We can, however, improve performance by not computing the hyperbolic functions in $U(x)$ from scratch when $k$ is small. Instead, we precompute $\exp(C)$ with the initial precision of $C$, and then compute $(\cosh(C/k), \sinh(C/k))$ from $(\exp(C))^{1/k}$; that is, by $k$th root extractions which have cost $O((\log k) M(r))$. Using the builtin MPFR functions, root extraction was found experimentally to be faster than evaluating the exponential function up to approximately $k = 35$ over a large range of precisions.

For extremely large $n$, we also speed up computation of the constant $C$ by using binary splitting to compute $\pi$ (adapting code written by H. Xue \cite{xue}) instead of the default function in MPFR, which uses arithmetic-geometric mean iteration.
As has been pointed out previously \cite{zim06}, binary splitting is more than four times faster for computing $\pi$ in practice, despite theoretically having a $\log$ factor worse complexity.  When evaluating $p(n)$ for multiple values of $n$, the value of $\pi$ should of course be cached, which MPFR does automatically.

\subsection{High-precision cosines}

The MPFR cosine and sine functions implement binary splitting, with similar asymptotics as the exponential function. At high precision, our implementation switches to custom code for evaluating $\alpha = \cos(p \pi/q)$ when $q$ is not too large, taking advantage of the fact that $\alpha$ is an algebraic number. Our strategy consists of generating a polynomial $P$ such that $P(\alpha) = 0$ and solving this equation using Newton iteration, starting from a double precision approximation of the desired root. Using a precision that doubles with each step of the Newton iteration, the complexity is $O(\operatorname{deg}(P) M(r))$.

The numbers $\cos(p \pi / q)$ are computed from scratch as needed: caching values with small $p$ and $q$ was found to provide a negligible speedup while needlessly increasing memory consumption and code complexity.

Our implementation uses the minimal polynomial $\Phi_n(x)$ of $\cos(2\pi/n)$, which has degree $d = \phi(n) / 2$ for $n \ge 3$ \cite{WZ93}. More precisely, we use the scaled polynomial $2^d \Phi(x) \in \mathbb{Z}[x]$. This polynomial is looked up from a precomputed table when $n$ is small, and otherwise is generated using a balanced product tree, starting from floating-point approximations of the conjugate roots. As a side remark, this turned out to be around a thousand times faster than computing the minimal polynomial with the standard commands in either Sage or Mathematica.

We sketch the procedure for high-precision evaluation of $\cos(p \pi/q)$ as Algorithm~\ref{alg:minpoly}, omitting various special cases and implementation details (for example,
our implementation performs the polynomial multiplications over $\mathbb{Z}[x]$ by embedding the approximate coefficients as fixed-point numbers).

\begin{algorithm}
\caption{High-precision numerical evaluation of $\cos(p \pi / q)$}
\label{alg:minpoly}
\begin{algorithmic}
\renewcommand{\algorithmicrequire}{\textbf{Input:}}
\renewcommand{\algorithmicensure}{\textbf{Output:}}
\REQUIRE Coprime integers $p$ and $q$ with $q \ge 3$, and a precision $r$
\ENSURE An approximation of $\cos(p \pi / q)$ accurate to $r$ bits
\STATE $n \gets (1 + (p \bmod 2)) \, q$
\STATE $d \gets \phi(n) / 2$
\STATE \COMMENT{Bound coefficients in $2^d \prod_{i=1}^d(x-\alpha)$}
\STATE $b \gets \left\lceil \log_2 d \right\rceil + \left\lceil \log_2 {d \choose d / 2} \right\rceil$
\STATE \COMMENT{Use a balanced product tree and a precision of $b + O(\log d)$ bits}
\STATE $\Phi \gets 2^d \prod_{i=1,\gcd(i,n)=1}^{\operatorname{deg}(\Phi) \le d} (x - \cos(i\pi/n))$ \COMMENT{Use basecase algorithm for cos}
\STATE \COMMENT{Round to an integer polynomial}
\STATE $\Phi \gets \sum_{k=0}^d \left\lfloor [x^k] \Phi + \frac{1}{2}\right\rfloor x^k$
\STATE Compute precisions $r_0 = r + 8, r_1 = r_0/2 + 8, \ldots, r_j = r_{j-1}/2+8 < 50$
\STATE $x \gets \cos(p \pi / q)$ \COMMENT{To 50 bits, using basecase algorithm}
\FOR {$i \gets j-1, j-2 \ldots 0$}
  \STATE \COMMENT{Evaluate using the Horner scheme at $r_i + b$ bit precision}
  \STATE $x \gets x - \Phi(x)/\Phi'(x)$
\ENDFOR
\RETURN $x$
\end{algorithmic}
\end{algorithm}

We do not attempt to prove that the internal precision management of Algorithm~\ref{alg:minpoly} is correct. However, the polynomial generation can easily be tested up to an allowed bound for $q$, and the function can be tested to be correct for all pairs $p, q$ at some fixed, high precision $r$. We may then argue heuristically that the well-behavedness of the object function in the root-finding stage combined with the highly conservative padding of the precision by several bits per iteration suffices to ensure full accuracy at the end of each step in the final loop, given an arbitrary $r$.

A different way to generate $\Phi_n(x)$ using Chebyshev polynomials is described in \cite{WZ93}. One can also use the squarefree part of an offset Chebyshev polynomial
$$P(x) = \frac{T_{2q}(x) - 1}{\gcd(T_{2q}(x) - 1, T'_{2q}(x))}$$
directly, although this is somewhat less efficient than the minimal polynomial.

Alternatively, since $\cos(p\pi/q)$ is the real part of a root of unity, the polynomial $x^q - 1$ could be used. The use of complex arithmetic adds overhead, but the method would be faster for large $q$ since $x^q$ can be computed in time $O((\log q) M(r))$ using repeated squaring. We also note that the secant method could be used instead of the standard Newton iteration in Algorithm \ref{alg:minpoly}. This increases the number of iterations, but removes the derivative evaluation, possibly providing some speedup.

In our implementation, Algorithm \ref{alg:minpoly} was found to be faster than the MPFR trigonometric functions for $q < 250$ roughly when the precision exceeds $400 + 4q^2$ bits. This estimate includes the cost of generating the minimal polynomial on the fly. 

\subsection{The main algorithm}

Algorithm \ref{alg:rademain} outlines the main routine in FLINT with only minor simplifications.
To avoid possible corner cases in the convergence of the HRR sum,
and to avoid unnecessary overhead, values with $n < 128$ (exactly corresponding to $p(n) < 2^{32}$) are looked up from a table. We only use $k$, $n$, $N$ in Theorem \ref{thm:prec} in order to
make the precision decrease uniformly,
allowing amortized summation to be implemented in a simple way.

\begin{algorithm}
\caption{Main routine implementing the HRR formula}
\label{alg:rademain}
\begin{algorithmic}
\renewcommand{\algorithmicrequire}{\textbf{Input:}}
\renewcommand{\algorithmicensure}{\textbf{Output:}}
\REQUIRE $n \ge 128$
\ENSURE $p(n)$
\STATE Determine $N$ and initial precision $r_1$ using Theorem \ref{thm:prec}
\STATE $C \gets \frac{\pi}{6} \sqrt{24n-1}$ \COMMENT{At $r_1 + 3$ bits}
\STATE $u \gets \exp(C)$
\STATE $s_1 \gets s_2 \gets 0$
\FOR {$1 \le k \le N$}
    \STATE Write term $k$ as \eqref{eq:termk} by calling Algorithm \ref{alg:fastsum}
    \IF {$A_k(n) \ne 0$}
        \STATE Determine term precision $r_k$ for $|t_k|$ using Theorem \ref{thm:prec}
        \STATE \COMMENT{Use Algorithm \ref{alg:minpoly} if $q_i < 250$ and $r_k > 400+4q^2$}
        \STATE $t \gets (-1)^s \sqrt{a/b}\prod \cos(p_i \pi / q_i)$
        \STATE $t \gets t \times U(C / k)$ \COMMENT{Compute $U$ from $u^{1/k}$ if $k < 35$}
        \STATE \COMMENT{Amortized summation: $r(s_2)$ denotes precision of the variable $s_2$}
        \STATE $s_2 \gets s_2 + t$
        \IF {$2 r_k < r(s_2)$}
            \STATE $s_1 \gets s_1 + s_2$ \COMMENT{Exactly or with precision exceeding $r_1$}
            \STATE $r(s_2) \gets r_k$ \COMMENT{Change precision}
            \STATE $s_2 \gets 0$
        \ENDIF
    \ENDIF
\ENDFOR
\RETURN $\lfloor s_1 + s_2 + \frac{1}{2}\rfloor$
\end{algorithmic}
\end{algorithm}

Since our implementation presently relies on some numerical heuristics (and in any case, considering the intricacy of the algorithm),
care has been taken to test it extensively. All $n \le 10^6$ have been checked explicitly, and a large number of isolated $n \gg 10^6$ have been compared against known congruences and values computed with Sage and Mathematica.

As a strong robustness check, we observe experimentally that the numerical error in the final sum decreases with larger $n$. For example, the error is consistently smaller than $10^{-3}$ for $n > 10^6$ and smaller than $10^{-4}$ for $n > 10^9$. This phenomenon reflects the fact that \eqref{eq:bound} overshoots the actual magnitude of the terms with large $k$, combined with the fact that rounding errors average out pseudorandomly rather than approaching worst-case bounds.

\section{Benchmarks}

Table \ref{tab:timings} and Figure \ref{fig:plot} compare performance of Mathematica 7, Sage 4.7 and FLINT on a laptop with a Pentium T4400 2.2 GHz CPU and 3 GiB of RAM, running 64-bit Linux. To the author's knowledge, Mathematica and Sage contain the fastest previously available partition function implementations by far.

The FLINT code was run with MPIR version 2.4.0 and MPFR version 3.0.1. Since Sage 4.7 uses an older version of MPIR and Mathematica is based on an older version of GMP, differences in performance of the underlying arithmetic slightly skew the comparison, but probably not by more than a factor two.

The limited memory of the aforementioned laptop restricted the range of feasible $n$ to approximately $10^{16}$. Using a system with an AMD Opteron 6174 processor and 256 GiB RAM allowed calculating $p(10^{17})$, $p(10^{18})$ and $p(10^{19})$ as well. The last computation took just less than 100 hours and used more than 150 GiB of memory, producing a result with over 11 billion bits. Some large values of $p(n)$ are listed in Table \ref{tab:values}.

\begin{table}
\begin{center}
\begin{tabular}{ | c | c | c | c | c | }
\hline
$n$ & Mathematica 7 & Sage 4.7 & FLINT & Initial \\
\hline
$10^4$ & 69 ms & 1 ms    & 0.20 ms & \\
$10^5$ & 250 ms & 5.4 ms & 0.80 ms & \\
$10^6$ & 590 ms & 41 ms & 2.74 ms & \\
$10^7$ & 2.4 s & 0.38 s & 0.010 s & \\
$10^8$ & 11 s & 3.8 s  & 0.041 s & \\
$10^9$ & 67 s & 42 s   & 0.21 s & 43\% \\
$10^{10}$ & 340 s &    & 0.88 s & 53\% \\
$10^{11}$ & 2,116 s &   & 5.1 s & 48\% \\
$10^{12}$ & 10,660 s &  & 20 s & 49\% \\
$10^{13}$ &  &         & 88 s & 48\% \\
$10^{14}$ &  &         & 448 s & 47\% \\
$10^{15}$ &  &         & 2,024 s & 39\% \\
$10^{16}$ &  &         & 6,941 s & 45\% \\
$10^{17}$ &  &         & 27,196* s & 33\% \\
$10^{18}$ &  &         & 87,223* s & 38\% \\
$10^{19}$ &  &         & 350,172* s & 39\% \\
\hline
\end{tabular}
\end{center}
\caption{Timings for computing $p(n)$ in Mathematica 7, Sage 4.7 and FLINT up to $n = 10^{16}$ on the same system, as well as FLINT timings for $n = 10^{17}$ to $10^{19}$ (*) done on different (slightly faster) hardware. Calculations running less than one second were repeated, allowing benefits from data caching. The rightmost column shows the amount of time in the FLINT implementation spent computing the first term.}
\label{tab:timings}
\end{table}


As can be seen in Table \ref{tab:timings} and Figure \ref{fig:plot}, the FLINT implementation exhibits a time complexity only slightly higher than $O(n^{1/2})$, with a comparatively small constant factor. The Sage implementation is fairly efficient for small $n$ but has a complexity closer to $O(n)$, and is limited to arguments $n < 2^{32} \approx 4 \times 10^9$.

\begin{figure}
\includegraphics[scale=0.64]{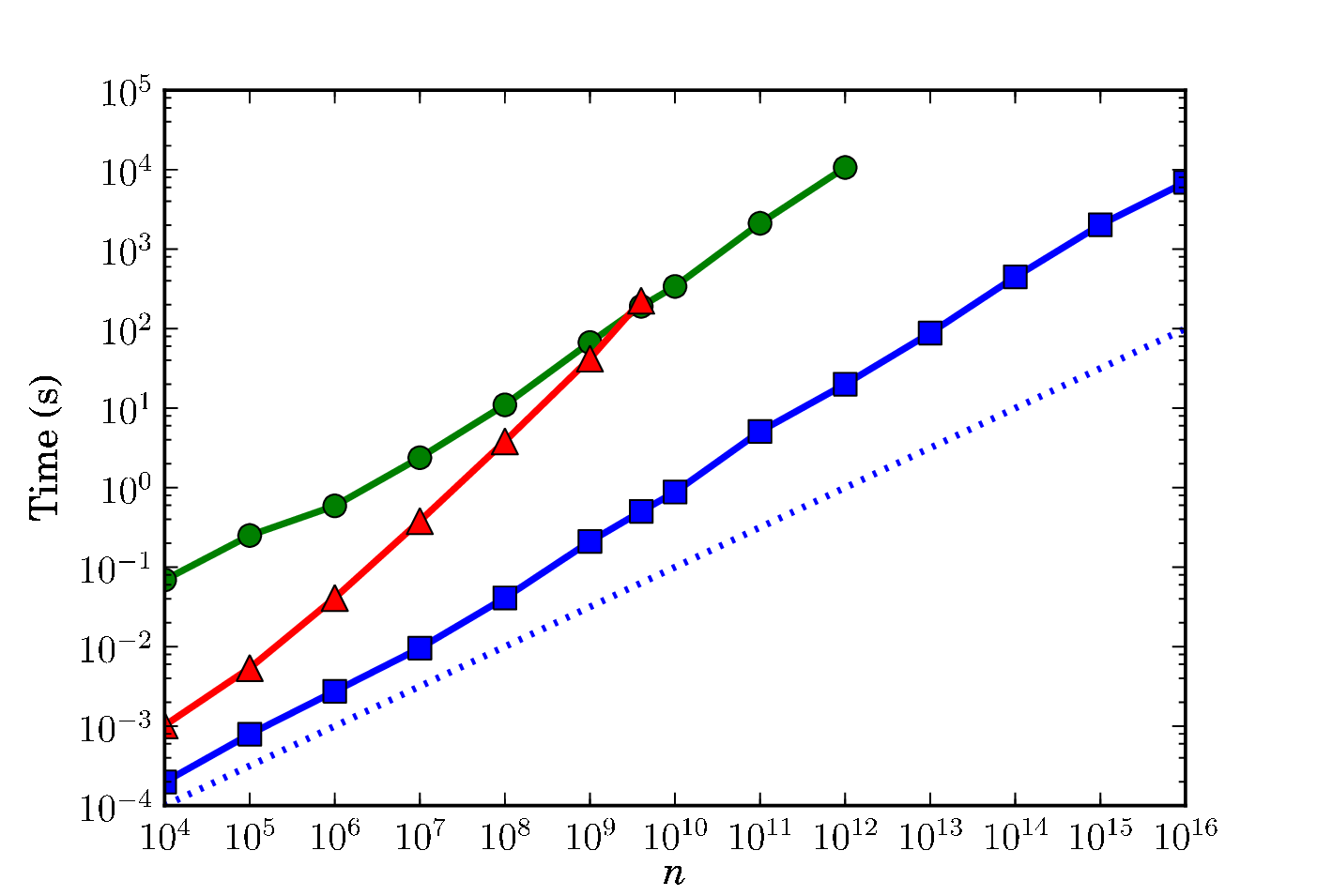}
\caption{CPU time $t$ in seconds for computing $p(n)$: FLINT (blue squares), Mathematica 7 (green circles), Sage 4.7 (red triangles). The dotted line shows $t = 10^{-6} n^{1/2}$, indicating
the slope of an idealized algorithm satisfying the trivial lower complexity bound $\Omega(n^{1/2})$ (the offset $10^{-6}$ is arbitrary).}
\label{fig:plot}
\end{figure}

The Mathematica implementation appears to have complexity slightly higher than $O(n^{1/2})$ as well, but consistently runs about 200--500 times slower than our implementation. Based on extrapolation, computing $p(10^{19})$ would take several years. It is unclear whether Mathematica is actually using a nearly-optimal algorithm or whether the slow growth is just the manifestation of various overheads dwarfing the true asymptotic behavior. The ratio compared to FLINT appears too large to be explained by differences in performance of the underlying arithmetic alone; for example, evaluating the first term in the series for $p(10^{10})$ to required precision in Mathematica only takes about one second.

We get one external benchmark from \cite{borwein2}, where it is reported that R. Crandall computed $p(10^9)$ in three seconds on a laptop in December 2008, ``using the Hardy-Ramanujan-Rademacher `finite' series for $p(n)$ along with FFT methods''. Even accounting for possible hardware differences, this appears to be an order of magnitude slower than our implementation.

\subsection{Optimality relative to the first term}

Table \ref{tab:timings} includes time percentages spent on evaluating the first term, $\exp(C)$, in the FLINT implementation. We find that this step fairly consistently amounts to just a little less than half of the running time. Our implementation is therefore nearly optimal in a practical sense, since the first term in the Hardy-Ramanujan-Rademacher expansion hardly can be avoided and at most a factor two can be gained by improving the tail evaluation.

\begin{table}
\begin{center}
\begin{tabular}{ | c | c | c | c | c | }
\hline
$n$ & Decimal expansion & Number of digits & Terms & Error \\
\hline
$10^{12}$ & $6129000962\ldots 6867626906$ & 1,113,996 & 264,526 & $2 \times 10^{-7}$ \\
$10^{13}$ & $5714414687 \ldots 4630811575$ & 3,522,791 & 787,010 & $3 \times 10^{-8}$  \\
$10^{14}$ & $2750960597 \ldots 5564896497$ & 11,140,072 & 2,350,465 & $-1 \times 10^{-8}$  \\
$10^{15}$ & $1365537729 \ldots 3764670692$ & 35,228,031 & 7,043,140 & $-3 \times 10^{-9}$  \\
$10^{16}$ & $9129131390 \ldots 3100706231$ & 111,400,846 & 21,166,305 & $-9 \times 10^{-10}$  \\
$10^{17}$ & $8291300791 \ldots 3197824756$ & 352,280,442 & 63,775,038 & $5 \times 10^{-10}$  \\
$10^{18}$ & $1478700310 \ldots 1701612189$ & 1,114,008,610 & 192,605,341 & $4 \times 10^{-10}$  \\
$10^{19}$ & $5646928403 \ldots 3674631046$ & 3,522,804,578 & 582,909,398 & $4 \times 10^{-11}$  \\
\hline
\end{tabular}
\end{center}
\caption{Large values of $p(n)$. The table also lists the number of terms $N$
in the Hardy-Ramanujan-Rademacher formula used by FLINT (theoretically bounding the error by 0.25) and the difference between the floating-point sum and the rounded integer.}
\label{tab:values}
\end{table}

Naturally, there is some potential to implement a faster version of the exponential function than the one provided by MPFR, reducing the cost of the first term. Improvements on the level of bignum multiplication would, on the other hand, presumably have a comparatively uniform effect.

By similar reasoning, at most a factor two can be gained through parallelization of our implementation by assigning terms in the Hardy-Ramanujan-Rademacher sum to separate threads. Further speedup on a multicore system requires parallelized versions of lower level routines, such as the exponential function or bignum multiplication. (A simplistic multithreaded version of the FLINT partition function was tested for $n$ up to $10^{18}$, giving nearly a twofold speedup on two cores, but failed when computing $10^{19}$ for reasons yet to be determined.) Fortunately, it is likely to be more interesting in practice to be able to evaluate $p(n)$ for a range of large values than just for a single value, and this task naturally parallelizes well.

\section{Multi-evaluation and congruence generation}

One of the central problems concerning the partition function
is the distribution of values of $p(n) \bmod m$.
In 2000, Ono
\cite{ono} proved that for every prime $m \ge 5$,
there exist infinitely many congruences of the type
\begin{equation}
p(Ak+B) \equiv 0 \bmod m
\label{eq:cong}
\end{equation}
where $A, B$ are fixed and $k$ ranges over all nonnegative integers.
Ono's proof is nonconstructive, but Weaver \cite{weaver} subsequently
gave an algorithm for finding congruences of this type when
$m \in \{13, 17, 19, 23, 29, 31\}$, and used the algorithm to compute
76,065 explicit congruences.

Weaver's congruences are specified by a tuple $(m,\ell,\varepsilon)$
where $\ell$ is a prime and $\varepsilon \in \{-1,0,1\}$, where
we unify the notation by writing $(m,\ell,0)$
in place of Weaver's $(m,\ell)$.
Such a tuple corresponds to a family of congruences
of the form \eqref{eq:cong} with coefficients
\begin{align}
\label{eq:Aeq}
A & = m \ell^{4-|\varepsilon|} \\
B & = \frac{m \ell^{3-|\varepsilon|} \alpha+1}{24} + m \ell^{3-|\varepsilon|} \delta,
\end{align}
where $\alpha$ is the unique solution of
$m \ell^{3-|\varepsilon|} \alpha \equiv -1 \bmod 24$
with $1 \le \alpha < 24$, and where $0 \le \delta < \ell$ is any solution of
\begin{align}
\begin{cases}
24 \delta \not\equiv -\alpha \bmod \ell
& \text{if } \varepsilon = 0 \\
\left(24\delta+\alpha \;| \;\ell \right) = \varepsilon
& \text{if } \varepsilon = \pm1.
\end{cases}
\label{eq:lsolv}
\end{align}

The free choice of $\delta$ gives $\ell - 1$
distinct congruences for a given tuple $(m,\ell,\varepsilon)$
if $\varepsilon = 0$, and $(\ell-1)/2$
congruences if $\varepsilon = \pm 1$.

Weaver's test for congruence, described by Theorems 7 and 8 in \cite{weaver}, essentially amounts to a single evaluation of $p(n)$ at a special point $n$. Namely, for given $m, \ell$, we compute the smallest solutions of $\delta_m \equiv 24^{-1} \bmod m$, $r_m \equiv -m \bmod 24$, and check whether $p(m r_m (\ell^2-1)/24 + \delta_m)$ is congruent mod $m$ to one of three values corresponding to the parameter $\varepsilon \in \{-1,0,1\}$. We give a compact statement of this procedure as Algorithm~\ref{alg:congruence}. To find new congruences, we simply perform a brute force search over a set of candidate primes $\ell$, calling Algorithm \ref{alg:congruence} repeatedly.

\begin{algorithm}
\caption{Weaver's congruence test}
\label{alg:congruence}
\begin{algorithmic}
\renewcommand{\algorithmicrequire}{\textbf{Input:}}
\renewcommand{\algorithmicensure}{\textbf{Output:}}
\REQUIRE A pair of prime numbers $13 \le m \le 31$ and $\ell \ge 5$, $m \ne \ell$
\ENSURE $(m,\ell,\varepsilon)$ defining a congruence, and Not-a-congruence otherwise
\STATE $\delta_m \gets 24^{-1} \bmod m$ \COMMENT{Reduced to $0 \le \delta_m < m$}
\STATE $r_m \gets (-m) \bmod 24$ \COMMENT{Reduced to $0 \le m < 24$}
\STATE $v \gets \frac{m - 3}{2}$
\STATE $x \gets p(\delta_m)$  \COMMENT {We have $x \not\equiv 0 \bmod m$}
\STATE $y \gets p\left(m \left(\frac{r_m (\ell^2 - 1)}{24} \right) + \delta_m \right)$
\STATE $f \gets (3 \;|\; \ell)\; ((-1)^{v} r_m \;|\; \ell)$ \COMMENT{Jacobi symbols}
\STATE $t \gets y + f x \ell^{v - 1}$
\IF {$t \equiv \omega \bmod m$ where $\omega \in \{-1,0,1\}$}
    \RETURN $(m, \ell, \omega \, (3 (-1)^{v}\; |\; \ell))$
\ELSE
    \RETURN Not-a-congruence
\ENDIF
\end{algorithmic}
\end{algorithm}

\subsection{Comparison of algorithms for vector computation}

In addition to the Hardy-Ramanujan-Rademacher formula, the
author has added code to FLINT for computing the vector of values
$p(0), p(1), \ldots p(n)$ over $\mathbb{Z}$ and
$\mathbb{Z}/m\mathbb{Z}$. The code is straightforward,
simply calling the default FLINT routines
for power series inversion over the respective coefficient rings,
which in both cases invokes Newton iteration and FFT multiplication
via Kronecker segmentation.

A timing comparison between the various methods
for vector computation is shown in Table~\ref{tab:seriescompare}.
The power series method is clearly the best choice for computing
all values up to $n$ modulo a fixed prime, having a complexity
of $O(n^{1+o(1)})$. For computing the full integer values, the power series
and HRR methods both have complexity $O(n^{3/2+o(1)})$,
with the power series method expectedly winning.

\begin{table}
\begin{center}
\begin{tabular}{ | c | c | c | c | c | }
\hline
$n$ & Series ($\mathbb{Z}/13\mathbb{Z}$) & Series ($\mathbb{Z}$) & HRR (all) & HRR (sparse) \\
\hline
$10^4$ & 0.01 s & 0.1 s & 1.4 s & 0.001 s \\
$10^5$ & 0.13 s & 4.1 s & 41 s & 0.008 s \\
$10^6$ & 1.4 s  & 183 s & 1430 s & 0.08 s \\
$10^7$ & 14 s   &       & & 0.7 s \\
$10^8$ & 173 s  &       & & 8 s   \\
$10^9$ & 2507 s &       & & 85 s  \\
\hline
\end{tabular}
\end{center}
\caption{Comparison of time needed to compute multiple values of $p(n)$ up to
the given bound, using power series inversion and the Hardy-Ramanujan-Rademacher
formula. The rightmost column gives the time when only computing the
subset of terms that are searched with Weaver's algorithm in the $m = 13$ case.}
\label{tab:seriescompare}
\end{table}

Ignoring logarithmic factors, we can expect the HRR formula
to be better than the power series for multi-evaluation of $p(n)$
up to some bound $n$ when $n / c$ values are needed. The factor $c \approx 10$ in the
FLINT implementation is a remarkable improvement
over $c \approx 1000$ attainable with previous implementations
of the partition function.
For evaluation mod $m$, the HRR formula is competitive when $O(n^{1/2})$
values are needed; in this case, the constant is highly sensitive to $m$.

For the sparse subset of $O(n^{1/2})$ terms
searched with Weaver's algorithm,
the HRR formula has the same complexity as the modular power series method,
but as seen in Table \ref{tab:seriescompare} runs more than an order of magnitude faster.
On top of this, it has the advantage of parallelizing trivially,
being resumable from any point, and requiring very little memory
(the power series evaluation mod $m = 13$ up to $n = 10^9$ required
over 40 GiB memory, compared to a few megabytes with the HRR formula).
Euler's method is, of course, also resumable from an arbitrary
point, but this requires computing and storing all previous values.

We mention that the authors of \cite{calkin} use a
parallel version of the recursive Euler method. This
is not as efficient as power series inversion,
but allows the computation to be split across multiple
processors more easily.

\subsection{Results}

Weaver gives 167 tuples, or 76,065 congruences, containing all
$\ell$ up to approximately 1,000--3,000 (depending on $m$).
This table was generated by computing all values of $p(n)$
with $n < 7.5 \times 10^6$
using the recursive version of Euler's pentagonal theorem.
Computing Weaver's table from scratch with FLINT,
evaluating only the necessary $n$, takes just a few seconds.
We are also able to numerically verify instances of
all entries in Weaver's table for small $k$.

As a more substantial exercise, we extend Weaver's table by determing
all $\ell$ up to $10^6$ for each
prime $m$. Statistics are listed in Table \ref{tab:results}.
The computation was performed by assigning
subsets of the search space to separate processes, running on
between 40 and 48 active cores for a period of four days,
evaluating $p(n)$ at $6 (\pi(10^6) - 3) = 470,970$ distinct $n$
ranging up to $2 \times 10^{13}$.

We find a total of 70,359 tuples, corresponding
to slightly more than $2.2~\times~10^{10}$ new congruences.
To pick an arbitrary, concrete example, one ``small'' new congruence is
$(13, 3797, -1)$ with $\delta = 2588$, giving
$$p(711647853449k + 485138482133) \equiv 0 \bmod 13$$
%
which we easily evaluate for all $k \le 100$, providing a sanity check on the identity as well as the partition function implementation. As a larger example, $(29, 999959, 0)$ with $\delta = 999958$
gives
$$p(28995244292486005245947069k + 28995221336976431135321047) \equiv 0 \bmod 29$$
which, despite our efforts, presently is out of reach for direct evaluation.

Complete tables of $(\ell,\varepsilon)$ for each $m$
are available at: \\
\url{http://www.risc.jku.at/people/fjohanss/partitions/} \\
\url{http://sage.math.washington.edu/home/fredrik/partitions/}

\begin{table}
\begin{center}
\begin{tabular}{ | c | c | c | c | c | c | c | }
\hline
$m$ & $(m,\ell,0)$ & $(m,\ell,+1)$ & $(m,\ell,-1)$ & Congruences & CPU & Max $n$ \\
\hline
$13$ & 6,189 & 6,000 & 6,132 & 5,857,728,831 & 448 h & $5.9 \times 10^{12}$ \\
$17$ & 4,611 & 4,611 & 4,615 & 4,443,031,844 & 391 h & $4.9 \times 10^{12}$ \\
$19$ & 4,114 & 4,153 & 4,152 & 3,966,125,921 & 370 h & $3.9 \times 10^{12}$ \\
$23$ & 3,354 & 3,342 & 3,461 & 3,241,703,585 & 125 h & $9.5 \times 10^{11}$  \\
$29$ & 2,680 & 2,777 & 2,734 & 2,629,279,740 & 1,155 h & $2.2 \times 10^{13}$  \\
$31$ & 2,428 & 2,484 & 2,522 & 2,336,738,093 & 972 h & $2.1 \times 10^{13}$  \\
\hline
All  & 23,376 & 23,367 & 23,616 & 22,474,608,014 & 3,461 h &  \\
\hline
\end{tabular}
\end{center}
\caption{The number of tuples of the given
type with $\ell < 10^6$, the total number
of congruences defined by these tuples, the total CPU time,
and the approximate bound up to which $p(n)$ was evaluated.}
\label{tab:results}
\end{table}

\section{Discussion}

Two obvious improvements to our implementation would be to develop a rigorous, and perhaps faster, version of Algorithm \ref{alg:minpoly} for computing $\cos(p \pi / q)$ to high precision, and to develop fast multithreaded implementations of transcendental functions to allow computing $p(n)$ for much larger $n$. Curiously, a particularly simple AGM-type iteration is known for $\exp(\pi)$ (see \cite{borwein}), and it is tempting to speculate whether a similar algorithm can be constructed for $\exp(\pi \sqrt{24n-1})$, allowing faster evaluation of the first term.

Some performance could also be gained with faster low-precision transcendental functions (up to a few thousand bits) and by using a better bound than \eqref{eq:bound} for the truncation error.

The algorithms described in this paper can be adapted to evaluation of
other HRR-type series, such as the number of partitions into distinct parts

\begin{equation}
Q(n) = \frac{\pi^2 \sqrt{2}}{24} \sum_{k=1}^{\infty} \frac{A_{2k-1}(n)}{(1-2k)^2}
\,_0F_1\left(2, \frac{(n+\frac{1}{24}) \pi^2}{12(1-2k)^2}\right).
\end{equation}

Using asymptotically fast methods for numerical evaluation of
hypergeometric functions, it should be possible to
retain quasi-optimality.

Finally, it remains an open problem whether there is a fast way to compute the isolated value $p(n)$ using purely algebraic methods. We mention the interesting recent work by Bruinier and Ono \cite{bo}, which perhaps could lead to such an algorithm.

\section{Acknowledgements}

The author thanks Silviu Radu for suggesting the application
of extending Weaver's table of congruences, and for explaining
Weaver's algorithm in detail. The author also thanks the anonymous referee for various suggestions, and Jonathan Bober for pointing out that Erd\H{o}s' theorem about quadratic nonresidues gives a rigorous complexity bound without assuming the Extended Riemann Hypothesis.

Finally, William Hart gave valuable feedback on various issues, and generously provided access to the computer hardware used for the large-scale computations reported in this paper. The hardware was funded by Hart's EPSRC Grant EP/G004870/1 (\mbox{Algorithms} in Algebraic Number Theory) and hosted at the University of Warwick.

\affiliationone{
   Fredrik Johansson\\
   Research Institute for Symbolic Computation\\
   Johannes Kepler University \\
   4040 Linz, Austria \\
   \email{fredrik.johansson@risc.jku.at}}

\end{document}